\documentclass[11pt]{article}
\usepackage{geometry}              
\usepackage{latexsym}
\usepackage{graphicx}
\usepackage{amssymb}
\usepackage{amsmath}
\usepackage{epstopdf}
\usepackage{mathtools}
\usepackage{tikz}
\usepackage{tikz-cd}
\usepackage{placeins}

\usepackage{mathrsfs}
\usepackage[cp850]{inputenc}
\usepackage{epsfig}
\usepackage{amsthm}
\usepackage{amscd}
\usepackage{amsfonts}
\usepackage{graphics}
\usepackage{graphicx}
\usepackage{float}

\newtheorem{theorem}{Theorem}
\newtheorem{lemma}[theorem]{Lemma}
 
\newtheorem{prop}[theorem]{Proposition}

\newtheorem{fact}[theorem]{Fact}
\newtheorem*{theorem*}{Theorem} 
\newtheorem*{corollary*}{Corollary}
\theoremstyle{definition}

\newtheorem{definition}[theorem]{Definition}
\newtheorem*{remark*}{Remark}

\newtheorem*{definition*}{Definition}
\newtheorem*{example*}{Example}

\newtheorem*{namedtheorem}{\theoremname}
\newcommand{\theoremname}{testing}
\newenvironment{named}[1]{\renewcommand{\theoremname}{#1}\begin{namedtheorem}}{\end{namedtheorem}}

\begin{document}

\title {{\bf Extending powers of pseudo-Anosovs}}
\author{ Cristina Mullican}
\date{}
\maketitle

\begin{quote}
{\small {\bf Abstract.} 
Biringer, Johnson, and Minsky showed that a pseudo-Anosov map on a boundary component of an irreducible 3-manifold has a power that partially extends to the interior if and only if the (un)stable laminations of $f$ is an $\mathbb{R}$-projective limit of meridians. We prove that the power required for a pseudo-Anosov map to partially extend is not universally bounded.  We construct a family of pseudo-Anosov maps $f_i$ for all $i=1,2...$ on a boundary component of a family of irreducible 3-manifolds $M_i$ such that $f_i^i$ partially extends to the interior of $M_i$ but $f_i^j$ does not for $j<i$. }
\end{quote}

\medskip 

\section{Introduction}

Let $M$ be a compact, orientable, and irreducible 3-manifold with some compressible boundary component $S$. Then for a homeomorphism $f:S\to S$, we say that $f$ \textit{partially extends} to $M$ if there is some nontrivial compression body $C\subset M$ with $\partial_+C=S$ and a homeomorphism $\phi:C\to C$ such that $\phi|_{S}=f$. 

Biringer, Johnson, and Minsky \cite{BJM} prove the following:

\begin{theorem*} 
[BJM, 2013] Let $f:\Sigma\to \Sigma$ be a pseudo-Anosov homeomorphism of some compressible boundary component $\Sigma$ of a compact, orientable and irreducible 3-manifold $M$. Then the (un)-stable lamination of $f$ is an $\mathbb{R}$-projective limit of meridians if and only if $f$ has a power that partially extends to $M$. 
\end{theorem*}

Ackermann gives an alternate proof in \cite{A} using earlier machinery of Casson and Long (see \cite{CL} and \cite {Long}). Maher and Schleimer give an alternate proof using train tracks and subsurface projections (see \cite{MS}). 

Partial extension comes up naturally when hyperbolizing 3-manifolds created as gluings. For instance, in \cite{L}, Lackenby studied the hyperbolization of 3-manifolds obtained by `generalized Dehn surgery', i.e. manifolds obtained by attaching a handlebody H to a compact 3-manifold M along some boundary component. He showed (modulo the Geometrization Theorem) that if M is `simple' and we choose any homeomorphism $\phi:\partial M\to \partial H$ and a homeomorphism $f : \partial H\to \partial H$ such that no power partially extends to $H$, then for infinitely many integers $n$, the manifold $M\cup_{f^n\circ \phi} H$ is hyperbolic.

Inspired by both of these theorems, it is natural to ask if there is some bound on the power of $f$ required to partially extend as this would imply there are only a finite number of powers of $f$ to check for partial extension and subsequent obstruction to hyperbolization. We will show that there is no universal bound via construction in the proof of the following:

\begin{theorem} For $i=1,2,...$ there is a compact, orientable and irreducible 3-manifold $M_i$ with compressible boundary component $\Sigma_i$ and a pseudo-Anosov $f_i:\Sigma_i\to \Sigma_i$ such that $f_i^i$ partially extends to $M_i$ and $f_i^j$ does not for $j<i$.
\label{T:main}
\end{theorem}

In section 2, we present background material and in section 3 we prove the main theorem via a construction. \\

\noindent {\bf Acknowledgments:} Many thanks to my advisor Ian Biringer for many helpful discussions and thoughtful guidance.

\section{Background}
A \textit{compression body} $C$ is an orientable, compact, irreducible 3-manifold with a preferred boundary component $\partial_+C$ that $\pi_1$-surjects. We say $\partial_+C$ is the \textit{exterior boundary component} of $C$.

\begin{definition} If we fix a surface $S$ to be the exterior boundary, an \textit{$S$-compression body} is a pair $(C,m)$ where $C$ is a compression body with a homeomorphism $m:S\to \partial_+ C$. Here, $m$ is the \textit{marking} of $C$ which is dropped when non-ambiguous. 
\end{definition}

Any $S$-compression body $C$ can be constructed in the following way: Set $\{\alpha_1,...,\alpha_n\}$ to be a maximal disjoint set of simple closed curves on $S$ that bound disks in $C$. First consider $S\times [0,1]$ and attach 2-handles along annuli in $S\times \{0\}$ whose core curves are $\{\alpha_i\}\times \{0\}$. Then we attach 3-balls along any resulting boundary components that are homeomorphic to $S^2$.  In this construction, $\partial_+C=S\times\{1\}$. For details see \cite{BV}. 

When an $S$-compression body can be constructed as above by attaching 2-handles along simple closed curves $\{\alpha_1,...,\alpha_n\}$ and subsequent 3-balls, we denote it as $S[\alpha_1,...,\alpha_n]$. The \textit{trivial} $S$-compression body is homeomorphic to $S\times I$. We call $\partial C\backslash \partial_+C$ the \textit{interior boundary} of $C$. 

For $S$-compression bodies $(C,m)$ and $(D,n)$ we write $(C,m)\subset(D,n)$ if there exists an embedding $H:C\to D$ such that $n=H\circ m$. 

Two $S$-compression bodies $(C,m)$ and $(C',m')$ are equivalent if there exists a homeomorphism $h:C\to C'$ such that the diagram below commutes.
\begin{center}
 \begin{tikzcd}[row sep=tiny]
                                     & \partial_+C'  \\
  S \arrow[ur,"m'"] \arrow[dr,"m"] &              \\
                                     &\partial_+C\arrow[uu,"h|_{\partial_+C}"']
 \end{tikzcd}
\end{center}

A homeomorphism $f:S\to S$ sends $(C,m)$ to $(C,mf^{-1} )$. This action respects the equivalence relation. From here forward, we will drop all markings and will abusively refer to compression bodies when we mean equivalence classes of compression bodies. In particular, we denote this action as $f(C)$.

\begin{definition} Let $C$ be an $S$-compression body and $f:S\to S$ a homeomorphism. We say that $f$ \textit{extends} to $C$ if there is a homeomorphism $\phi:C\to C$ such that $\phi|_{\partial_+C}=f$. Here we can write $f(C)=C$. Recall that we say $f$ \textit{partially extends} to $C$ if there is an $S$-compression body $D\subset C$ with a homeomorphism $\psi:D\to D$ such that $\psi|_{\partial_+D}=f$.
\end{definition}

\begin{definition}
Let $\alpha$ be a simple closed curve in $\partial_+ C$. If $\alpha$ bounds a disk in $C$ then we say $\alpha$ is a \textit{meridian} of $C$ and $\alpha$ \textit{compresses} in $C$. Moreover, if there is some simple closed curve $\alpha'$ in the interior boundary of $C$ such that $\alpha$ and $\alpha'$ together bound an embedded annulus in $C$ we say $\alpha$ \textit{bounds an annulus}.
\end{definition}

Recall the following well known fact which is proved in detail in \cite{M}:
\begin{fact} Let $C$ be an $S$-compression body. A Dehn twist $T_\alpha:S\to S$ extends to $C$ if $\alpha$ compresses in $C$ or bounds an annulus in $C$. 
\label{Fact:dehntwist}
\end{fact}

The idea of the proof is to define a Dehn twist of $C$  by twisting in a neighborhood of the disk or annulus bounded by $\alpha$.

In the compression body $S[\alpha_1,...,\alpha_n]$ there are likely many curves besides the $\alpha_i$ that compress. 

\begin{fact} If two boundary components of an embedded pair of pants in $\partial_+C$ bound disks then the third boundary component also bounds a disk. \label{Fact:pants}\end{fact}

 \begin{proof} Let $P$ be a pair of pants embedded in $\partial_+C$ with boundary components $c_1,c_2$ and $c_3$ such that $c_1,c_2$ bounding disks $d_1$ and $d_2$ respectively.  Since  $P\cup d_1 \cup d_2$  is homeomorphic to a disk then $c_3 $ compresses. 
 \end{proof}

\begin{definition} An $S$-compression body $C$ is \textit{small} if it can be written as $S[a]$ for some simple closed curve $a\in S=\partial_+C$. A compression body is \textit{minimal} if it does not contain any (non-trivial) sub-compression bodies. 
\end{definition}

Minimality and smallness are related in the following work of Biringer and Vlamis. 

\begin{prop}[\cite{BV}, Cor. 2.7] An $S$-compression body is minimal if and only if it is a solid torus or a small compression body obtained by compressing a separating curve. 
\label{Prop:small}
\end{prop}

A compression body can be built out of minimal compression bodies in the following way: 

\begin{definition} A \textit{sequence of minimal compressions} of an $S$-compression body $C$ is a chain $S\times [0,1]=C_0\subset C_1\subset\cdots\subset C_k=C$ of $S$-compression bodies where $C_{i+1}$ is obtained from $C_i$ by gluing  in a minimal $F_i$-compression body to $F_i$, an interior boundary component of $C_i$. \cite{BV} \end{definition}

In this sequence, each compression body is obtained by gluing in either a solid torus or obtained by compressing a single separating curve. Biringer and Vlamis give a formula for the number of steps required to obtain a compression body from minimal compressions:

\begin{prop}[\cite{BV}, Prop. 2.10] If $C$ is any $S$-compression body with interior boundary $F_1\sqcup \cdots\sqcup F_n$, then the length $k$ of any sequence of minimal  compressions $S\times [0,1]=C_0\subset C_1\subset \cdots \subset C_k=C$ is $\mathfrak{h}(C):=2g(S)-1-\sum_{i=1}^n (2g(F_i)-1)$, the \textnormal{height} of $C$.
\label{Prop:height}
\end{prop}


\section{Main Theorem}

We will prove Theorem 1 by constructing a family of manifolds and corresponding pseudo-Anosovs. In fact, the manifolds we construct below are compression bodies. We restate our Theorem 1 to reflect this:

\begin{named}{Theorem \ref{T:main}}
\textit{For $g=1,2,...$ there is an $S_{2g}$-compression body $C_g$ and a pseudo-Anosov $f_g:S_{2g}\to S_{2g}$ such that $f_g^g$ extends to $C_g$ and $f_g^j$ does not partially extend to $C_g$ for $j<g$.}
\end{named}

\begin{proof} Fix $2g$, the genus of the exterior boundary component, and consider the compression body $K_1:=S_{2g}[\gamma,\alpha]$ as shown in Figure \ref{Fig:cbK1}.  We will construct a pseudo-Anosov homeomorphism $f_g:S_{2g}\to S_{2g}$ such that $f^g_g$ extends to $K_1$ but $f_g^j$ does not extend to any sub-compression body of $K_1$ for $j\leq g$. 

\begin{figure}[H]
\centering
\includegraphics[height=2.4 in]{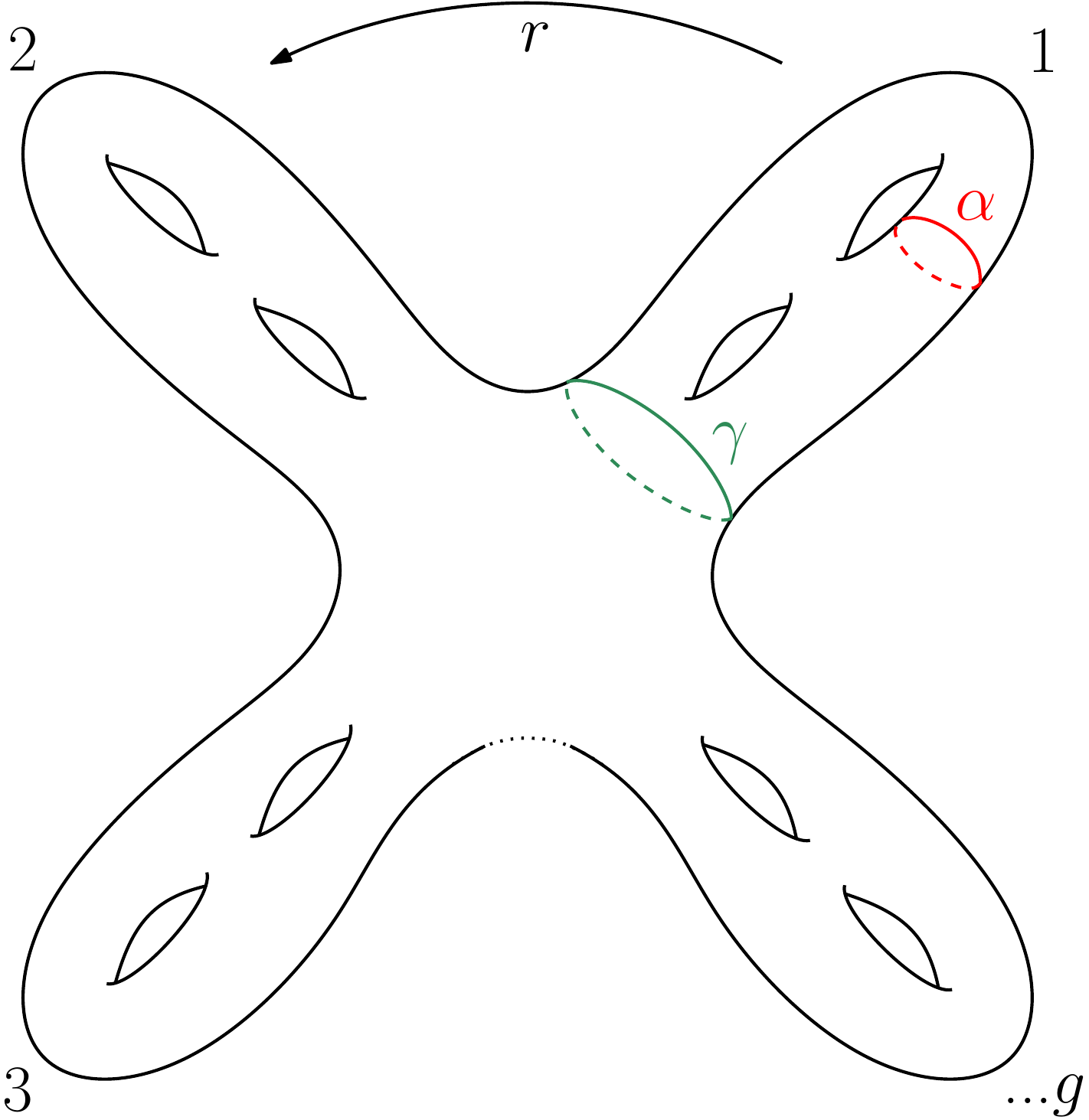}
\caption{Compression body $K_1:=S_{2g}[\gamma,\alpha]$}
\label{Fig:cbK1}
\end{figure}

Let $r$ be the rotation of $S_{2g}$ by $\frac{2\pi}{g}$. Define $K_j:=r^{j-1}(K_1)$ for $1<j\leq g$. Since $r^g$ is the identity then $r^g$ extends to $K_1$. Thus $r(K_i)=K_{(i+1)\mbox{mod}(g)}$.

\begin{lemma} Each curve in the set  $\mathcal{K}=\{\sigma,\varphi,\beta,\gamma, \alpha,b_1,b_2,b_3\}$  (see Figure \ref{Fig:setofcurves}) bounds a disk or annulus in $K_1,...,K_g$. 
\label{Lem:diskannulus}
\end{lemma}

\begin{figure}[h]
\centering
\includegraphics[width=4 in]{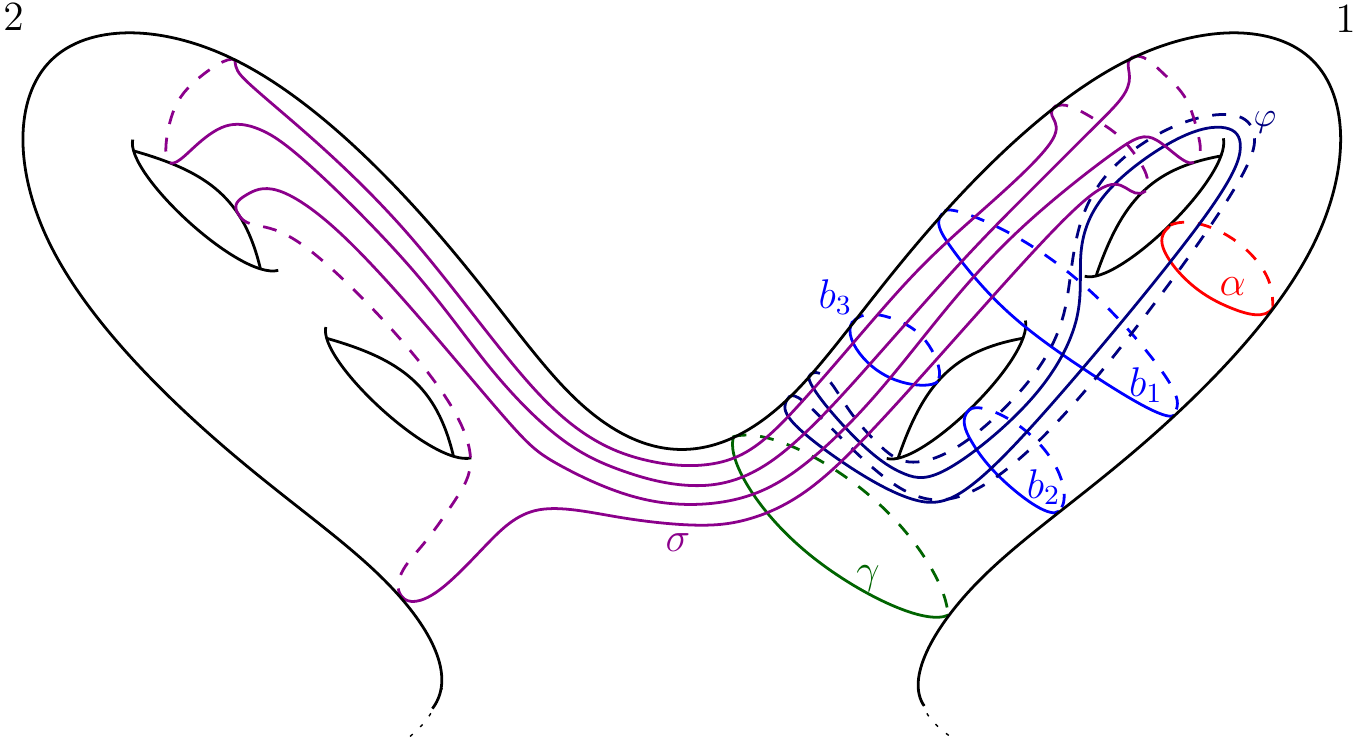}
\caption{Set $\mathcal{K}=\{\sigma,\varphi,\beta, \gamma, \alpha,b_1,b_2,b_3\}$.}
\label{Fig:setofcurves}
\end{figure}
\begin{proof}
First, note that every curve in $\mathcal{K}$ bounds an annulus in $K_i$ for $i>2$ and every curve in $\mathcal{K}\backslash \{\sigma\}$ bounds an annulus in $K_2$. 

Clearly $\alpha$ and $\gamma$ bound disks in $K_1$. Also $b_2$ and $b_3$ bound annuli. Since $\alpha$ and $b_1$ co-bound a pair of pants, by Fact \ref{Fact:pants}, $b_1$ also bounds a disk in $K_1$. Also by Fact \ref{Fact:pants}, curve $\varphi'$ bounds a disk in $K_1$ as seen in Figure \ref{Fig:phi}. Notice that $\varphi$ is a band-sum of this disk. Thus $\varphi$ also bounds a disk in $K_1$.

\begin{figure}[h]
\centering
\includegraphics[height=2 in]{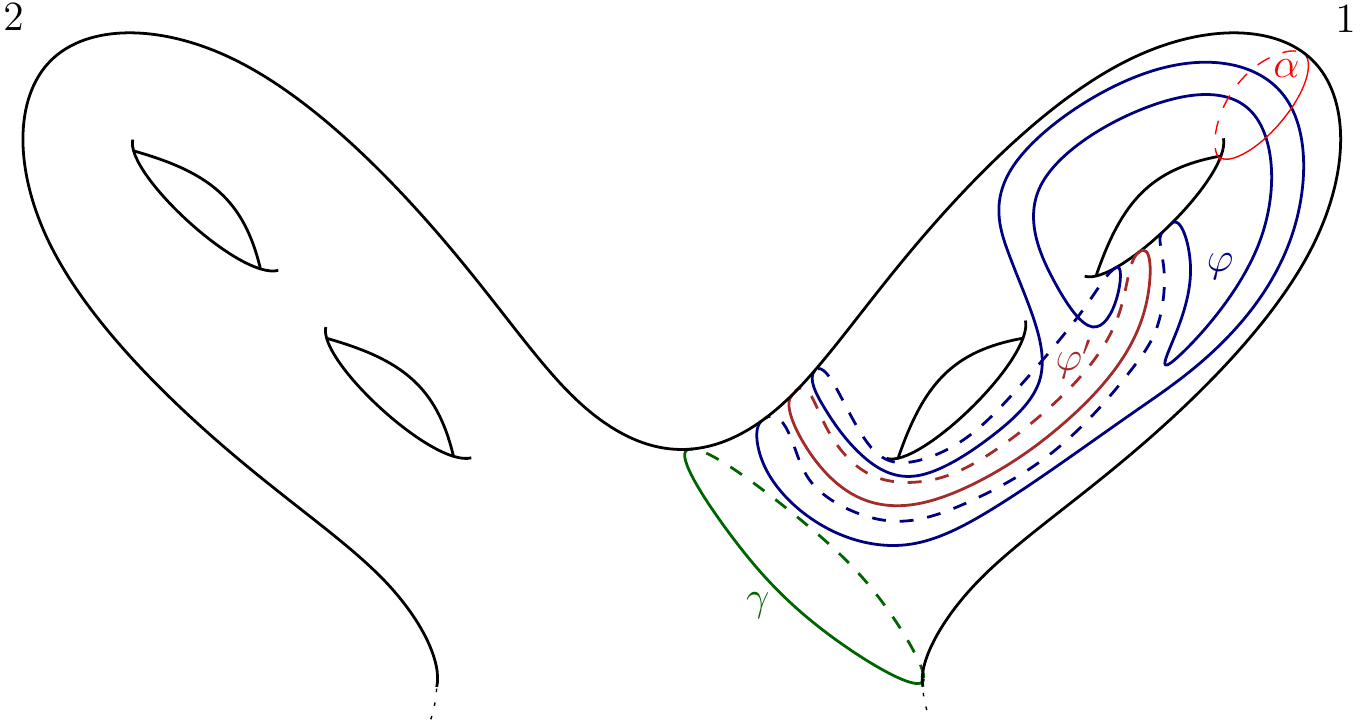}
\caption{Curve $\varphi$, a bandsum of $\varphi'$}
\label{Fig:phi}
\end{figure}

In Figure \ref{Fig:sigmaK1} in steps (a) through (d) we see that $\sigma$ homotopes to $\tilde{\sigma_1}$ in $K_1$ which bounds an annulus. Hence $\sigma$ bounds an annulus in $K_1$. In Figure \ref{Fig:sigmaK2} we see that in $K_2$, $\sigma$ is homotopic to $\tilde{\sigma_2}$ which co-bounds a pair of pants with  $r(\gamma)$ and $r(\alpha)$ which both compress in $K_2$. Thus $\sigma$ bounds a disk in $K_2$.  

\begin{figure}[h]
\centering
\includegraphics[width=4.5 in]{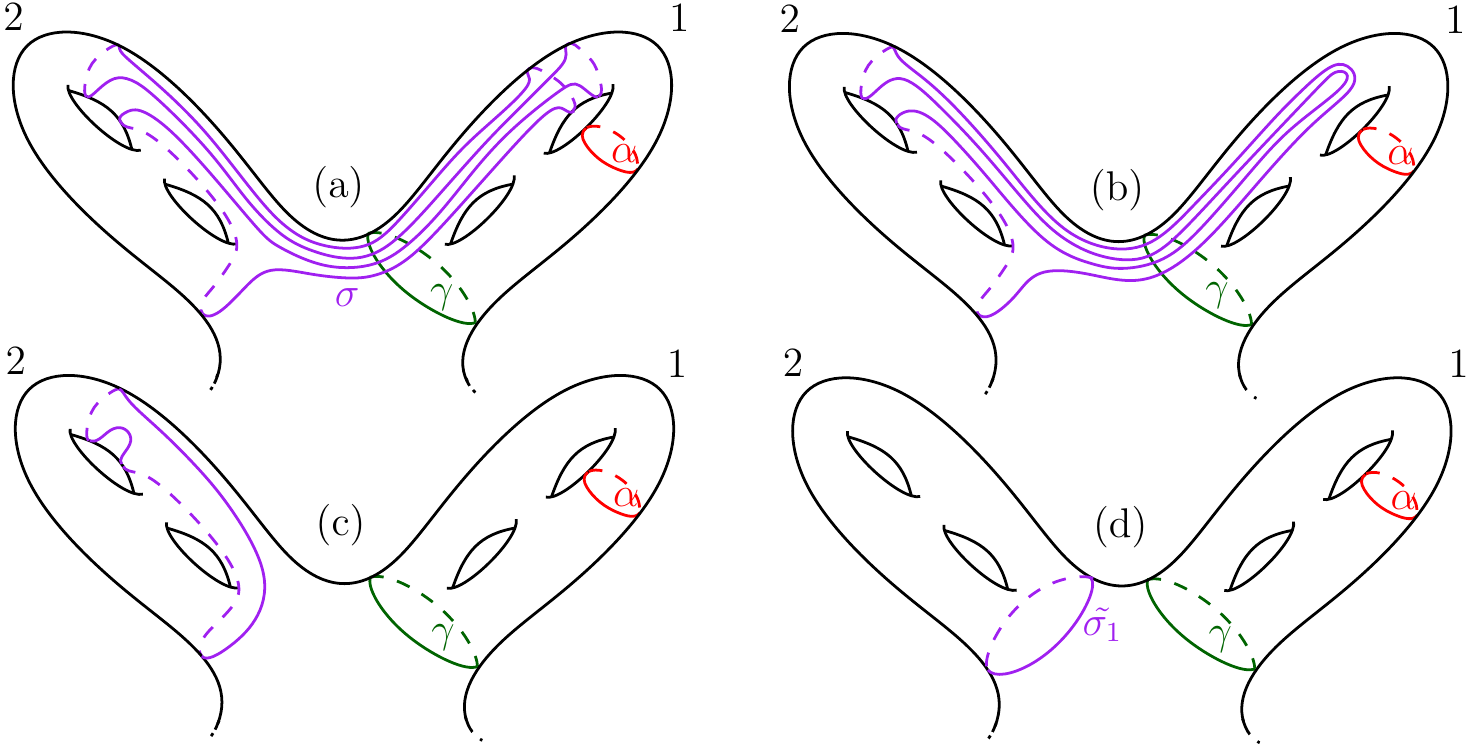}
\caption{Curve $\sigma$ bounds an annulus in $K_1$}
\label{Fig:sigmaK1}
\end{figure}

\begin{figure}[h]
\centering
\includegraphics[width=6 in]{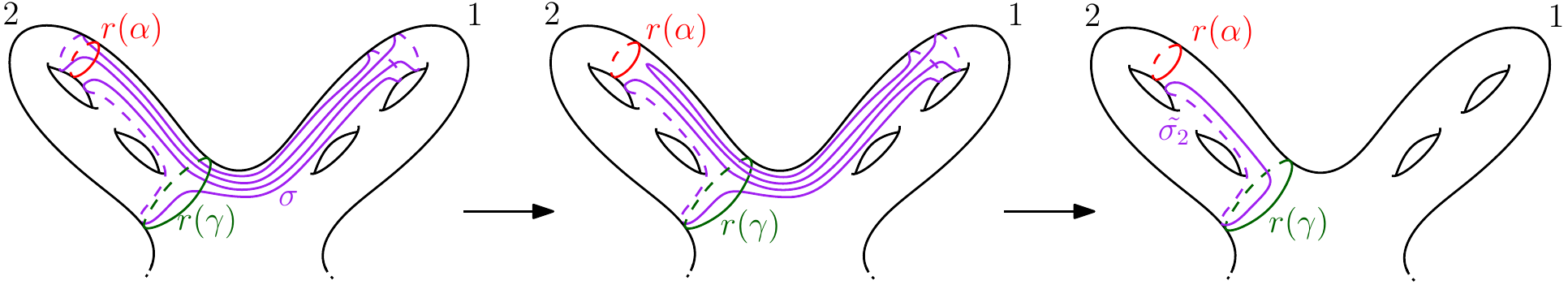}
\caption{Curve $\sigma$ bounds a disk in $K_2$}
\label{Fig:sigmaK2}
\end{figure}
\end{proof}

We will apply the following theorem of Fathi: 
\begin{theorem*}[\cite{F}, Theorem 0.2] Let $h$ be a mapping class of surface a $S$ and $\gamma_1,...,\gamma_k$ be simple closed curves on $S$. Suppose that the orbits under $h$ of the $\gamma_i$ are distinct and fill $S$. Then there exists an $n\in \mathbb{N}$ such that for every $(n_1,...,n_k)\in \mathbb{Z}^k$ with $|n_i|\geq n$, the class $T^{n_k}_{\gamma_k}...T^{n_1}_{\gamma_1}h$ is pseudo-Anosov.
\end{theorem*}

The curves $\bigcup_{i=1}^{g} r^i(\mathcal{K})$ fill $S_{2g}$ as follows. Looking at Figure 2, we see that $\alpha, \varphi,\sigma, b_1,b_2,$ and $b_3 $ fill the genus 2 component of $S_{2g}\backslash \gamma$. Also, note that the genus 0 component of $S_{2g}\backslash \bigcup_{i=0}^{g-1} r^i(\gamma)$ is filled by $\bigcup_{i=0}^{g-1} r^i(\sigma)$. Hence the orbit of $\mathcal{K}$ under $r$ fills $S_{2g}$. 

By Fathi's theorem there exists $n_{g_1},\cdots, n_{g_7}\in \mathbb{Z}$ such that $$T_{\gamma}^{n_{g_1}}\circ T_{\alpha}^{n_{g_2}}\circ T_{\sigma}^{n_{g_3}}\circ T_{\varphi}^{n_{g_4}}\circ T_{b_1}^{n_{g_5}}\circ T_{b_2}^{n_{g_6}}\circ T_{b_3}^{n_{g_7}}\circ r:=f_g$$ is pseudo-Anosov. 

\begin{lemma} The homeomorphism $f_g^g:S_{2g}\to S_{2g}$ extends to $K_1$ but $f_g^j$ does not extend to $K_1$ for $j<g$. 
\label{Lem:extend}
\end{lemma}

\begin{proof} Recall that $r(K_i)=K_{(i+1)\mbox{mod}(g)}$. By Lemma \ref{Lem:diskannulus}, every curve in $\mathcal{K}$ bounds a disk or annulus so applying Proposition \ref{Fact:dehntwist} implies that $T_c$ extends to $K_i$  for all $c\in \mathcal{K}$ and $1\leq i\leq g$. Therefore, $f_g(K_i)=K_{(i+1)\mbox{mod}(g)}$. Hence, $f_g^g(K_1)=K_1$ and so $f_g^g$ extends to $K_1$ but $f_g^j$ does not for $j<g$. 
\end{proof}

We must now show that $f^j$ for $j<g$ does not \textit{partially} extend to $K_1$. First, we prove two lemmas.

\begin{lemma} Every common meridian of $K_i$ and $K_j$ for $i\neq j$ is separating.
\label{Lem:separating}
\end{lemma}
\begin{proof} Let's first consider the homology of $K_i$. Set $\alpha_i:=r^{i-1}(\alpha)$ and $\gamma_i:=r^{i-1}(\gamma)$. Hence both $\alpha_i$ and $\gamma_i$ compress in $K_i$. Let $D_{\alpha_i}$ and $D_{\gamma_i}$ denote the disks in $K_i$ bounded by $\alpha_i$ and $\gamma_i$ respectively. Consider the following portion of the the Mayer-Vietoris long exact sequence: 
$$H_1(S_{2g}\cap (D_{\alpha_i}\sqcup D_{\gamma_i}))\xrightarrow{\Phi_i}H_1(S_{2g})\oplus H_1(D_{\alpha_i}\sqcup D_{\gamma_i})\xrightarrow{\Psi_i}H_1(K_i).$$

Note that $H_1(D_{\alpha_i}\sqcup D_\gamma)=0$ giving $H_1(S_{2g})\xrightarrow{\Psi_i}H_1(K_i)$. Also, $S_{2g}\cap (D_{\alpha_i}\sqcup D_{\gamma_i})=\{\alpha_i,\gamma_i\}$. By exactness of the Mayor-Vietoris sequence, $ \mbox{Im}(\Phi_i)=\Phi_i(\{\alpha_i,\gamma_i\})=\ker(\Psi_i)$. Moreover, $\gamma_i$ is separating and hence $\Phi_i([\gamma_i])=0$ in $H_1(S_{2g})$. Therefore, $\ker(\Psi_i)$ is generated by $\alpha_i$. 

Note that the $[\alpha_i]$ for $i\in \{1,2,...,g\}$ form a subset of a basis for $H_1(S_{2g})$. If $\omega$ is a simple closed curve in $S_{2g}$ and a meridian of both $K_i$ and $K_j$ then $[\omega]\in \ker(\Psi_i)\cap\ker(\Psi_j)\subset H_1(S_{2g})$. So $[\omega]\in \langle[\alpha_i]\rangle\cap \langle[\alpha_j]\rangle$ implying $[\omega]=0$ in $H_1(S_{2g})$. Therefore, $\omega$ is a separating curve of $S_{2g}$. 
\end{proof}

\begin{lemma} Every $S$-compression body $D\subset K_i\cap K_j$ is small.
\label{Lem:small}
\end{lemma}
\begin{proof} For all $i$, the interior boundary components of $K_i$ are homeomorphic to $S_1\sqcup S_{2g-2}$. Applying Proposition \ref{Prop:height}, for all $i$, the height of $K_i$ is 

\begin{eqnarray*}
\mathfrak{h}(K_i)&=&(2(2g)-1)-[2(1)-1]-[2(2g-2)-1]\\
&=&4g-1-1-(4g-5)\\
&=&3.
\end{eqnarray*}

Let $D$ be an $S$-compression body with $D\subset K_i\cap K_j$. Then $\mathfrak{h}(D)<3$. 

If $\mathfrak{h}(D)=1$ then $S_{2g}\times [0,1]=C_0\subsetneq C_1=D$ is a sequence of minimal compressions for $D$. By Proposition \ref{Prop:small}, since $D$ is not a solid torus, $D$ is small.

If  $\mathfrak{h}(D)=2$ then $S_{2g}\times [0,1]=C_0\subsetneq C_1\subsetneq C_2=D$ is a sequence of minimal compressions for $D$. By Lemma \ref{Lem:separating}, both compressions are along separating curves. Therefore, the interior boundary components of $D$ are $F_1\sqcup F_2\sqcup F_3$ where $g(F_1)+g(F_2)+g(F_3)=2g$ and $g(F_\ell)>0$ for $\ell\in\{1,2,3\}$. Recall that the interior boundary components of $K_i$ are $S_{2g-2}\sqcup S_1$. Then $D\subset K_i$ implies that $g(F_\ell)\geq 2g-2$ for some $\ell\in\{1,2,3\}$. Without loss of generality, say $g(F_1)\geq 2g-2$. This forces $g(F_1)=2g-2$, $g(F_2)=1$ and $g(F_3)=1$. 

There is some sequence of minimal compressions with $S_{2g}\times[0,1]=C_0\subsetneq C_1\subsetneq C_2=D\subsetneq C_3=K_i$. Thus we must compress a curve in $F_2$ or $F_3$ to obtain $K_i$ (to preserve the genus $2g-2$ interior boundary component). Likewise, we must compress a curve in $F_2$ or $F_3$ to obtain $K_j$. Therefore, $K_i$ and $K_j$ share $F_1$, the same genus $2g-2$ interior boundary component. But this is false by construction of $K_i$ and $K_j$. Therefore, $\mathfrak{h}(D)\neq 2$. 

Thus, $D$ is a small compression body.
\end{proof}

To prove Theorem \ref{T:main} recall we must show that for all $j<g$, $f_g^j$ does not \textit{partially} extend to $K_1$. Assume for sake of contradiction that there is some $S_{2g}$-compression body $K'\subsetneq K_1$ with $\ell<g$ such that $f_g^\ell$ extends to $K'$. Then $f_g^{\ell} (K')=K'$. Since $f_g^\ell(K_1)=K_{\ell+1}$, this implies that $K'\subset K_{\ell+1}$. Thus $K'\subset K_1\cap K_{\ell+1}$. By Lemma \ref{Lem:small}, $K'$ must be a small $S_{2g}$-compression body and by Lemma \ref{Lem:separating}, $K'$ has exactly one meridian, $a$, a separating curve. Then $f_g^\ell$ must map the disk bounded by $a$ to itself. This implies that $f_g^\ell(a)=a$ which contradicts the fact that $f_g^\ell$ is pseudo-Anosov.

Therefore, $f^g$ extends to $K_1$ and no lesser power partially extends to $K_1$. As the genus to goes to infinity, the power required for $f$ to extend to the corresponding compression body $K_1$ also goes to infinity. 
\end{proof}

\end{document}